\documentclass{amsart}%
\usepackage{graphicx}
\usepackage{amscd}
\usepackage{amsmath}
\usepackage{amsfonts}
\usepackage[unicode=true]{hyperref}
\usepackage{amssymb}%
\providecommand{\U}[1]{\protect\rule{.1in}{.1in}}
\newtheorem{theorem}{Theorem}
\theoremstyle{plain}

\newtheorem{corollary}{Corollary}

\newtheorem{example}{Example}

\newtheorem{remark}{Remark}

\numberwithin{equation}{section}

\begin{document}
\title[Addendum BMO]{Addendum to \textquotedblleft BMO: OSCILLATIONS, SELF IMPROVEMENT, GAGLIARDO
COORDINATE SPACES AND REVERSE HARDY INEQUALITIES"}
\author{Mario Milman}
\email{mario.milman@gmail.com}
\urladdr{https://sites.google.com/site/mariomilman}

\begin{abstract}
We provide a precise statement and self contained proof of a Sobolev
inequality (cf. \cite[ page 236 and page 237]{A}) stated in the original
paper. Higher order and fractional inequalities are treated as well.

\end{abstract}
\maketitle

\section{Introduction}

One of the purposes of the original paper (cf. \cite{A}) was to highlight some
connections between interpolation theory, and inequalities connected with the
theory of $BMO$ and Sobolev spaces. This resulted in a somewhat lengthy paper
and as consequence many known results were only stated, and the reader was
referred to the relevant literature for proofs. It has become clear, however,
that a complete account of some of the results could be useful. In this
expository addendum we update and correct one paragraph of the original text
by providing a precise statement and proof of a Sobolev inequality which was
stated in the original paper (cf. \cite[(13) page 236, and line 10, page
237]{A}). Included as well are the corresponding results for higher order and
fractional inequalities.

All the results discussed in this note are known\footnote{In presenting the
results yet again we have followed in part advise from Rota \cite{82}.}: The
only novelty is perhaps in the unified presentation.

We shall follow the notation and the ordering of references of the original
paper \cite{A} to which we shall also refer for background, priorities,
historical comments, etc. Newly referenced papers will be labeled with letters.

\section{The Hardy-Littlewood-Sobolev-O'Neil program}

We let%
\begin{equation}
\left\Vert f\right\Vert _{L(p,q)}=\left\{
\begin{array}
[c]{cc}%
\left\{  \int_{0}^{\infty}\left(  f^{\ast}(t)t^{1/p}\right)  ^{q}\frac{dt}%
{t}\right\}  ^{1/q} & 1\leq p<\infty,1\leq q\leq\infty\\
\left\Vert f\right\Vert _{L(\infty,q)} & 1\leq q\leq\infty,
\end{array}
\right.  \label{berta}%
\end{equation}
where%
\begin{equation}
\left\Vert f\right\Vert _{L(\infty,q)}:=\left\{  \int_{0}^{\infty}(f^{\ast
\ast}(t)-f^{\ast}(t))^{q}\frac{dt}{t}\right\}  ^{1/q}. \label{berta1}%
\end{equation}
In particular we note that in this notation%
\[
\left\Vert f\right\Vert _{L(\infty,\infty)}=\{f:\sup_{t>0}\{f^{\ast\ast
}(t)-f^{\ast}(t)\}<\infty\},
\]%
\[
L(1,1)=L^{1}.
\]
Moreover, if $f$ has bounded support%
\[
\left\Vert f\right\Vert _{L(\infty,1)}=\left\Vert f\right\Vert _{L^{\infty}}.
\]
In \cite[(13) page 236]{A} we stated that \textquotedblleft it was shown in
\cite{7} that
\begin{equation}
\left\Vert f\right\Vert _{L(\bar{p},q)}\leq c_{n}\left\Vert \nabla
f\right\Vert _{L(p,q)},1\leq p\leq n,\frac{1}{\bar{p}}=\frac{1}{p}-\frac{1}%
{n},\text{ }1\leq q\leq\infty,f\in C_{0}^{\infty}(R^{n})." \label{sobolev}%
\end{equation}
However, to correctly state what was actually shown in \cite{7}, the indices
in the displayed formula need to be restricted when $p=1$. The precise
statement reads as follows:

\begin{theorem}
\label{teo1}Let $n>1.$ Let $1\leq p\leq n,$ $1\leq q\leq\infty,$ and define
$\frac{1}{\bar{p}}=\frac{1}{p}-\frac{1}{n}.$ Then, if $(p,q)\in(1,n]\times
\lbrack1,\infty]$ \textbf{or if} $p=q=1,$ we have%
\begin{equation}
\left\Vert f\right\Vert _{L(\bar{p},q)}\leq c_{n}(p,q)\left\Vert \left\vert
\nabla f\right\vert \right\Vert _{L(p,q)},\text{ }f\in C_{0}^{\infty}(R^{n}).
\label{sobol1}%
\end{equation}

\end{theorem}

\begin{remark}
If $n=1,$ then $p=q=1,$ and (\ref{sobol1}) is an easy consequence of the
fundamental theorem of Calculus.
\end{remark}

The corresponding higher order result (cf. \cite[line 10, page 237]{A}) reads
as follows,

\begin{theorem}
\label{teo2}Let $k\in N$, $k\leq n,$ $1\leq p\leq\frac{n}{k},$ $1\leq
q\leq\infty.$ Define $\frac{1}{\bar{p}}=\frac{1}{p}-\frac{k}{n}.$ Then, (i) if
$k<n,$ and $(p,q)\in(1,\frac{n}{k}]\times\lbrack1,\infty],$ or $(p,q)\in
\{1\}\times\{1\},$ or (ii) if $n=k,$ and $p=q=1,$ we have%
\[
\left\Vert f\right\Vert _{L(\bar{p},q)}\leq c_{n,k}(p,q)\left\Vert \left\vert
D^{k}f\right\vert \right\Vert _{L(p,q)},\text{ }f\in C_{0}^{\infty}(R^{n}),
\]
where $\left\vert D^{k}f\right\vert $ is the length of the vector whose
components are all the partial derivatives of order $k.$
\end{theorem}

\begin{remark}
Observe that when $p=\frac{n}{k},$ $p>1$, and $q=1,$ we have
\[
\left\Vert f\right\Vert _{L^{\infty}}=\left\Vert f\right\Vert _{L(\infty
,1)}\preceq\left\Vert \left\vert D^{k}f\right\vert \right\Vert _{L(\frac{n}%
{k},1)},\text{ }f\in C_{0}^{\infty}(R^{n}).
\]
We also obtain an $L^{\infty}$ estimate when $p=\frac{n}{k}=1,$ and $q=1,$%
\begin{equation}
\left\Vert f\right\Vert _{L^{\infty}}=\left\Vert f\right\Vert _{L(\infty
,1)}\preceq\left\Vert D^{n}f\right\Vert _{L^{1}},\text{ }f\in C_{0}^{\infty
}(R^{n}). \label{deacuerdo}%
\end{equation}

\end{remark}

In the particular case when we are working with $L^{p}$ spaces, i.e. $p=q$,
\textbf{there is no need to separate the cases }$p=1$\textbf{ and }$p>1,$ and
Theorems \ref{teo1} and \ref{teo2} give us what we could call the
\textquotedblleft completion" of the Hardy-Littlewood-Sobolev-O'Neil program, namely

\begin{corollary}
Let $1\leq k\leq n,1\leq p\leq\frac{n}{k},\frac{1}{\bar{p}}=\frac{1}{p}%
-\frac{k}{n}.$ Then%
\begin{equation}
\left\Vert f\right\Vert _{L(\bar{p},p)}\leq c_{n}(p)\left\Vert D^{k}%
f\right\Vert _{L(p,p)},\text{ }f\in C_{0}^{\infty}(R^{n}). \label{sobsup}%
\end{equation}

\end{corollary}

\begin{proof}
\textbf{(of Theorem \ref{teo1})}. \textbf{The case }$1<p\leq n.$ We start with
the inequality (cf. \cite[(58) page 263]{A}),%
\begin{equation}
f^{\ast\ast}(t)-f^{\ast}(t)\leq c_{n}t^{1/n}(\nabla f)^{\ast\ast}(t),
\label{sob2}%
\end{equation}
which yields%
\begin{equation}
\left(  f^{\ast\ast}(t)-f^{\ast}(t)\right)  t^{1/p}t^{-1n}\leq c_{n}%
t^{1/p}(\nabla f)^{\ast\ast}(t). \label{sobol2}%
\end{equation}
If $q<\infty,$ we integrate (\ref{sobol2}) and find%
\begin{align*}
\left\{  \int_{0}^{\infty}[\left(  f^{\ast\ast}(t)-f^{\ast}(t)\right)
t^{1/\bar{p}}]^{q}\frac{dt}{t}\right\}  ^{1/q}  &  \leq c_{n}\left\{  \int%
_{0}^{\infty}[t^{1/p}(\nabla f)^{\ast\ast}(t)]^{q}\frac{dt}{t}\right\}
^{1/q}\\
&  \leq C_{n}(p,q)\left\Vert \nabla f\right\Vert _{L(p,q)},
\end{align*}
where in the last step we used Hardy's inequality (cf. \cite[Appendix 4, page
272]{St}). To identify the left hand side we consider two cases. If $p=n,$
then $\bar{p}=\infty$ and the desired result follows directly from the
definitions (cf. (\ref{berta})). If $p<n,$ then we can write%
\begin{equation}
f^{\ast\ast}(t)=\int_{t}^{\infty}f^{\ast\ast}(s)-f^{\ast}(s)\frac{ds}{s},
\label{from}%
\end{equation}
and use Hardy's inequality (cf. \cite[Appendix 4, page 272]{St}) to get%
\[
\left\Vert f\right\Vert _{L(\bar{p},q)}\leq\left\{  \int_{0}^{\infty}%
[f^{\ast\ast}(t)t^{1/\bar{p}}]^{q}\frac{dt}{t}\right\}  ^{1/q}\preceq\left\{
\int_{0}^{\infty}[\left(  f^{\ast\ast}(t)-f^{\ast}(t)\right)  t^{1/\bar{p}%
}]^{q}\frac{dt}{t}\right\}  ^{1/q}.
\]
The case $q=\infty$ is easier. Indeed, if $p=n,$ the desired result follows
taking a sup in (\ref{sobol2})$,$ while if $p<n,$ from (\ref{from}) we find%
\begin{align*}
f^{\ast\ast}(t)  &  \leq\int_{t}^{\infty}\left(  f^{\ast\ast}(s)-f^{\ast
}(s)\right)  s^{1/\bar{p}}s^{-1/\bar{p}}\frac{ds}{s}\\
&  \preceq t^{-1/\bar{p}}\sup_{s}\left(  f^{\ast\ast}(s)-f^{\ast}(s)\right)
s^{1/\bar{p}}.
\end{align*}
Consequently%
\[
\left\Vert f\right\Vert _{L(\bar{p},\infty)}\preceq\sup_{s}\left(  f^{\ast
\ast}(s)-f^{\ast}(s)\right)  s^{1/\bar{p}}.
\]
Therefore, combining the estimates we have obtained for the right and left
hand sides, we obtain%
\[
\left\Vert f\right\Vert _{L(\bar{p},\infty)}\preceq\left\Vert \nabla
f\right\Vert _{L(p,q)},1<p\leq n,1\leq q\leq\infty.
\]

Finally, we consider the case when\textbf{ }$p=q=1.$ In this case we have
$\frac{1}{\bar{p}}=1-\frac{1}{n}.$ At this point recall the inequality (cf.
\cite[page 264]{A})%
\begin{align}
\int_{0}^{t}\left(  f^{\ast\ast}(s)-f^{\ast}(s)\right)  s^{1/\bar{p}}\frac
{ds}{s}  &  =\int_{0}^{t}\left(  f^{\ast\ast}(s)-f^{\ast}(s)\right)
s^{1-1/n}\frac{ds}{s}\\
&  \preceq\int_{0}^{t}(\nabla f)^{\ast}(s)ds. \label{dav1}%
\end{align}
Let $t\rightarrow\infty,$ to find%
\[
\int_{0}^{\infty}\left(  f^{\ast\ast}(s)-f^{\ast}(s)\right)  s^{1/\bar{p}%
}\frac{ds}{s}\preceq c_{n}\int_{0}^{\infty}(\nabla f)^{\ast}(s)ds=c_{n}%
\left\Vert \nabla f\right\Vert _{L^{1}}=c_{n}\left\Vert \nabla f\right\Vert
_{L(1,1)}.
\]
We conclude the proof remarking that, as we have seen before,%
\[
\left\Vert f\right\Vert _{L(\bar{p},1)}\preceq\int_{0}^{\infty}\left(
f^{\ast\ast}(s)-f^{\ast}(s)\right)  s^{1/\bar{p}}\frac{ds}{s}.
\]

\end{proof}

\section{Higher Order}

We will only deal in detail with the case $k=2$ (i.e. the case of second order
derivatives) since the general case follows by induction,\textit{ mutatis
mutandi}.

\begin{proof}
(i) Suppose first that $n>2.$ Let $\bar{p}_{1}$ and $\bar{p}_{2}$ be defined
by $\frac{1}{\bar{p}_{1}}=\frac{1}{p}-\frac{1}{n}$ and $\frac{1}{\bar{p}_{2}%
}=\frac{1}{\bar{p}_{1}}-\frac{1}{n}=\frac{1}{p}-\frac{2}{n}=\frac{1}{\bar{p}%
}.$ The first step of the iteration is to observe (cf. \cite{75}) the
elementary fact:%
\[
\left\vert \nabla(\nabla f)\right\vert \leq\left\vert D^{2}(f)\right\vert .
\]
Therefore, by (\ref{sobol1}) we have%
\begin{align*}
(\nabla f)^{\ast\ast}(t)-(\nabla f)^{\ast}(t)  &  \preceq t^{1/n}%
[\nabla(\nabla f)]^{\ast\ast}(t)\\
&  \preceq t^{1/n}\left\vert D^{2}(f)\right\vert ^{\ast\ast}(t).
\end{align*}
Consequently, we find%
\begin{equation}
\left(  (\nabla f)^{\ast\ast}(t)-(\nabla f)^{\ast}(t)\right)  t^{1/\bar{p}%
_{1}}\preceq t^{\frac{1}{p}}\left\vert D^{2}(f)\right\vert ^{\ast\ast}(t).
\label{sob4}%
\end{equation}
Suppose that $1<p\leq\frac{n}{2},$ and let $1\leq q<\infty.$ Then, from
(\ref{sob4}) and a familiar argument, we get
\begin{align*}
\left\Vert \nabla f\right\Vert _{L(\bar{p}_{1},q)}  &  \preceq\left\{
\int_{0}^{\infty}[\left(  (\nabla f)^{\ast\ast}(t)-(\nabla f)^{\ast
}(t)\right)  t^{1/\bar{p}_{1}}]^{q}\frac{dt}{t}\right\}  ^{1/q}\\
&  \preceq\left\{  \int_{0}^{\infty}[t^{1/p}\left\vert D^{2}(f)\right\vert
^{\ast\ast}(t)]^{q}\frac{dt}{t}\right\}  ^{1/q}.
\end{align*}
Thus,%
\[
\left\Vert \nabla f\right\Vert _{L(\bar{p}_{1},q)}\preceq\left\Vert \left\vert
D^{2}(f)\right\vert \right\Vert _{L(p,q)}.
\]
Now, combining the previous inequality with the already established first
order case (cf. Theorem \ref{teo1}) we find,%
\begin{align*}
\left\Vert f\right\Vert _{L(\bar{p}_{2},q)}  &  \preceq\left\Vert \nabla
f\right\Vert _{L(\bar{p}_{1},q)}\\
&  \preceq\left\Vert \left\vert D^{2}(f)\right\vert \right\Vert _{L(p,q)}.
\end{align*}
Likewise we can treat the case when $q=\infty.$ The analysis also works in the
case $p=1=q.$ In this case we replace (\ref{sob4}) with (\ref{dav1}):
\[
\int_{0}^{t}\left(  \nabla f)^{\ast\ast}(s)-(\nabla f)^{\ast}(s)\right)
s^{1-1/n}\frac{ds}{s}\preceq\int_{0}^{t}(D^{2}f)^{\ast}(s)ds,
\]
which yields%
\[
\int_{0}^{\infty}\left(  \nabla f)^{\ast\ast}(s)-(\nabla f)^{\ast}(s)\right)
s^{1-1/n}\frac{ds}{s}\preceq\int_{0}^{\infty}(D^{2}f)^{\ast}(s)ds.
\]
Therefore%
\[
\left\Vert \nabla f\right\Vert _{L(\bar{p}_{1},1)}\preceq\int_{0}^{\infty
}(D^{2}f)^{\ast}(s)ds.
\]
At this point recall that the first order case gives us%
\[
\left\Vert f\right\Vert _{L(\bar{p}_{2},1)}\preceq\left\Vert \nabla
f\right\Vert _{L(\bar{p}_{1},1)}.
\]
Thus,%
\[
\left\Vert f\right\Vert _{L(\bar{p}_{2},1)}\preceq\left\Vert D^{2}f\right\Vert
_{L^{1}}.
\]

Finally consider the case when $n=2=$ $k,$ this means that $p=\frac{2}{2}=1,$
and we let $q=1.$ Then, from%
\[
\int_{0}^{t}\left(  (Df)^{\ast\ast}(s)-\left(  Df\right)  ^{\ast}(s)\right)
s^{1-1/2}\frac{ds}{s}\preceq\int_{0}^{t}(D^{2}f)^{\ast}(s)ds
\]
we once again derive%
\[
\left\Vert \nabla f\right\Vert _{L(2,1)}\preceq\left\Vert D^{2}f\right\Vert
_{L^{1}}.
\]
Moreover, since
\[
\left(  f^{\ast\ast}(t)-f^{\ast}(t)\right)  \preceq t^{1/2}\left(  \nabla
f\right)  ^{\ast\ast}(t)
\]
integrating we get%
\[
\left\Vert f\right\Vert _{L(\infty,1)}\preceq\left\Vert \nabla f\right\Vert
_{L(2,1)},
\]
consequently, we see that,%
\[
\left\Vert f\right\Vert _{L^{\infty}}=\left\Vert f\right\Vert _{L(\infty
,1)}\preceq\left\Vert D^{2}f\right\Vert _{L^{1}}.
\]

\end{proof}

\begin{example}
In the case $n>2,p=\frac{n}{2},$ $q=1,$ we have%
\[
\left\Vert f\right\Vert _{L(\infty,1)}\preceq\left\Vert \nabla f\right\Vert
_{L(n,1)}\preceq\left\Vert D^{2}f\right\Vert _{L(\frac{n}{2},1)},
\]
in other words%
\begin{equation}
\left\Vert f\right\Vert _{L^{\infty}}\preceq\left\Vert D^{2}f\right\Vert
_{L(\frac{n}{2},1)}. \label{steine}%
\end{equation}

\end{example}

\begin{remark}
Sobolev inequalities involving only the Laplacian are usually referred to as
*reduced Sobolev inequalities* and there is a large literature devoted to
them. For example, in the context of the previous Example, since $n/2>1$ it is
possible to replace $D^{2}$ by the Laplacian in (\ref{steine}) (cf. the
discussion in \cite[Chapter V]{St}). The correct *reduced* analog of
(\ref{steine}) when $n=2$ involves a stronger condition on the Laplacian, as
was recently shown by Steinerberger \cite{Stef}, who, in particular, shows
that for a domain $\Omega\subset R^{2}$ of finite measure, and $f\in
C^{2}(\Omega)\cap C(\bar{\Omega}),$ there exists an absolute constant $c>0$
such that%
\[
\max_{x\in\Omega}\left\vert f(x)\right\vert \leq\max_{x\in\partial\Omega
}\left\vert f(x)\right\vert +c\max_{x\in\Omega}\int_{\Omega}\max\{1,\log
\frac{\left\vert \Omega\right\vert }{\left\vert x-y\right\vert ^{2}%
}\}\left\vert \Delta f(y)\right\vert dy.
\]
In particular, when $f$ is zero at the boundary, Steinerberger's result gives%
\begin{equation}
\max_{x\in\Omega}\left\vert f(x)\right\vert \leq c\max_{x\in\Omega}%
\int_{\Omega}\max\{1,\log\frac{\left\vert \Omega\right\vert }{\left\vert
x-y\right\vert ^{2}}\}\left\vert \Delta f(y)\right\vert dy. \label{steine2}%
\end{equation}
By private correspondence Steinerberger showed the author that (\ref{steine2})
implies an inequality of the form%
\begin{equation}
\left\Vert f\right\Vert _{L^{\infty}(\Omega)}\preceq\left\Vert \Delta
f\right\Vert _{L^{1}(\Omega)}+\left\Vert \Delta f\right\Vert _{L(LogL)(\Omega
)}. \label{steine3}%
\end{equation}

Let us informally put forward here that one can develop an approach to
Steinerberger's result (\ref{steine3}) using the symmetrization techniques of
this paper, if one uses a variant of symmetrization inequalities for the
Laplacian, originally obtained by Maz'ya-Talenti, that was recorded in
\cite[Theorem 13 (ii), page 178]{Mm}. We hope to give a detailed discussion elsewhere.
\end{remark}

\section{The Fractional Case}

In this section we remark that a good deal of the analysis can be also adapted
to the fractional case (cf. \cite{59}). Let us go through the details. Let
$X(R^{n})$ be a rearrangement invariant space, and let $\phi_{X}(t)=\left\Vert
\chi_{(0,t)}\right\Vert _{X},$ be its fundamental function. Let $w_{X}$ be the
modulus of continuity associated with $X:$%
\[
w_{X}(t,f)=\sup_{\left\vert h\right\vert \leq t}\left\Vert f(\circ
+h)-f(\circ)\right\Vert _{X}.
\]
Our basic inequality will be (cf. [50] and [59])%
\begin{equation}
f^{\ast\ast}(t)-f^{\ast}(t)\leq c_{n}\frac{w_{X}(t^{1/n},f)}{\phi_{X}(t)},f\in
C_{0}^{\infty}(R^{n}). \label{nueva}%
\end{equation}
Let $\alpha\in(0,1),1\leq p\leq\frac{n}{\alpha},$ $1\leq q\leq\infty.$ Let
(with the usual modification if $q=\infty)$%
\[
\left\Vert f\right\Vert _{\mathring{B}_{p}^{\alpha,q}}=\left\{  \int%
_{0}^{\infty}[t^{-\alpha}w_{L^{p}}(t,f)]^{q}\frac{dt}{t}\right\}  ^{1/q}.
\]

\begin{theorem}
Suppose that $\alpha\in(0,1),1\leq p\leq\frac{n}{\alpha},\frac{1}{\bar{p}%
}=\frac{1}{p}-\frac{\alpha}{n}.$ Then, we have
\[
\left\Vert f\right\Vert _{L(\bar{p},q)}\preceq\left\Vert f\right\Vert
_{\mathring{B}_{p}^{\alpha,q}},f\in C_{0}^{\infty}(R^{n}).
\]

\end{theorem}

\begin{proof}
Consider first the case $q<\infty.$ Let $X=L^{p},$ then $\phi_{X}(t)=t^{1/p},$
consequently (\ref{nueva}) becomes%
\[
f^{\ast\ast}(t)-f^{\ast}(t)\leq c_{n}\frac{w_{L^{p}}(t^{1/n},f)}{t^{1/p}},f\in
C_{0}^{\infty}(R^{n}),
\]
which yields%
\begin{align*}
\left\{  \int_{0}^{\infty}[(f^{\ast\ast}(t)-f^{\ast}(t))t^{\frac{1}{\bar{p}}%
}]^{q}\frac{dt}{t}\right\}  ^{1/q} &  =\left\{  \int_{0}^{\infty}[(f^{\ast
\ast}(t)-f^{\ast}(t))t^{-\alpha/n}t^{1/p}]^{q}\frac{dt}{t}\right\}  ^{1/q}\\
&  \leq c_{n}\left\{  \int_{0}^{\infty}[t^{-\alpha/n}w_{L^{p}}(t^{1/n}%
,f)]^{q}\frac{dt}{t}\right\}  ^{1/q}\\
&  \simeq\left\{  \int_{0}^{\infty}[t^{-\alpha}w_{L^{p}}(t,f)]^{q}\frac{dt}%
{t}\right\}  ^{1/q}\\
&  \simeq\left\Vert f\right\Vert _{\mathring{B}_{p}^{\alpha,q}}.
\end{align*}
It follows readily that%
\[
\left\Vert f\right\Vert _{L(\bar{p},q)}\preceq\left\Vert f\right\Vert
_{\mathring{B}_{p}^{\alpha,q}},\text{ }f\in C_{0}^{\infty}(R^{n}).
\]
For the case $q=\infty$ we simply go back to
\begin{equation}
f^{\ast\ast}(t)-f^{\ast}(t))t^{\frac{1}{\bar{p}}}\leq c_{n}t^{-\alpha/n}%
w_{p}(t^{1/n},f),\label{antigua}%
\end{equation}
and take a sup over all $t>0$.
\end{proof}

\begin{example}
Note that when $p=\frac{n}{\alpha},$ then $\frac{1}{\bar{p}}=0,$ consequently
if $1\leq q\leq\infty$, we have that for $f\in C_{0}^{\infty}(R^{n}),$
\begin{align}
\left\Vert f\right\Vert _{L(\infty,q)}  &  =\left\{  \int_{0}^{\infty
}[(f^{\ast\ast}(t)-f^{\ast}(t))]^{q}\frac{dt}{t}\right\}  ^{1/q}
\label{nueva2}\\
&  \leq c_{n}\left\Vert f\right\Vert _{\mathring{B}_{\frac{n}{\alpha}}%
^{\alpha,q}}.\nonumber
\end{align}
In particular, if $q=1,$%
\[
\left\Vert f\right\Vert _{L^{\infty}}=\left\Vert f\right\Vert _{L(\infty
,1)}\leq c_{n}\left\Vert f\right\Vert _{\mathring{B}_{\frac{n}{\alpha}%
}^{\alpha,1}},f\in C_{0}^{\infty}(R^{n}).
\]

\end{example}

The corresponding result for Besov spaces anchored on Lorentz spaces follows
the same analysis. Let $1\leq p<\infty,1\leq r\leq\infty,1\leq q\leq
\infty,0<\alpha<1.$ We let (with the usual modification if $q=\infty$)
\[
\left\Vert f\right\Vert _{\mathring{B}_{L(p,r)}^{\alpha,q}}=\left\{  \int%
_{0}^{\infty}[t^{-\alpha}w_{L(p,r)}(t,f)]^{q}\frac{dt}{t}\right\}  ^{1/q}.
\]
Note that since
\[
\phi_{L(p,r)}(t)\sim t^{1/p},1\leq p<\infty,1\leq r\leq\infty,
\]
our basic inequality now takes the form%
\begin{equation}
f^{\ast\ast}(t)-f^{\ast}(t)\leq c_{n}\frac{w_{L(p,r)}(t^{1/n},f)}{t^{1/p}%
},f\in C_{0}^{\infty}(R^{n}),1\leq p<\infty,1\leq r\leq\infty. \label{denueva}%
\end{equation}
Then,\textit{ mutatis mutandi} we have

\begin{theorem}
Suppose that $\alpha\in(0,1),1\leq p\leq\frac{n}{\alpha},\frac{1}{\bar{p}%
}=\frac{1}{p}-\frac{\alpha}{n}.$ Then, if $p>1,1\leq r\leq\infty,$ or $p=r=1,$
we have
\[
\left\Vert f\right\Vert _{L(\bar{p},q)}\preceq\left\Vert f\right\Vert
_{\mathring{B}_{L(p,r)}^{\alpha,q}},f\in C_{0}^{\infty}(R^{n}).
\]

\end{theorem}

\section{More Examples and Remarks}

\subsection{On the role of the $L(\infty,q)$ spaces}

In the range $1<p<n,$ (\ref{sobol1}) and (\ref{sobsup}) yield the classical
Sobolev inequalities. Suppose that $p=n.$ Then $\frac{1}{\bar{p}}=0,$ and
(\ref{sobol1}) becomes%
\begin{equation}
\left\Vert f\right\Vert _{L(\infty,q)}\preceq\left\Vert \nabla f\right\Vert
_{L(n,q)},1\leq q\leq\infty. \label{hbr}%
\end{equation}
When dealing with domains $\Omega$ with $\left\vert \Omega\right\vert
<\infty,$ from (\ref{sob2}) we get, $1\leq q\leq\infty,$%
\begin{equation}
\left\{  \int_{0}^{\left\vert \Omega\right\vert }\left(  f^{\ast\ast
}(s)-f^{\ast}(s)\right)  ^{q}\frac{ds}{s}\right\}  ^{1/q}\preceq\left\Vert
\nabla f\right\Vert _{L(n,q)},\text{ }f\in C_{0}^{\infty}(\Omega).
\label{hbr1}%
\end{equation}
To compare this result with classical results it will be convenient to
normalize the *norm* as follows%
\[
\left\Vert f\right\Vert _{L(\infty,q)(\Omega)}=\left\{  \int_{0}^{\left\vert
\Omega\right\vert }\left(  f^{\ast\ast}(s)-f^{\ast}(s)\right)  ^{q}\frac
{ds}{s}\right\}  ^{1/q}+\frac{1}{\left\vert \Omega\right\vert }\int_{\Omega
}\left\vert f(x)\right\vert dx.
\]

\begin{remark}
Note that this does not change the nature of (\ref{hbr1}) since if $f$ has
compact support on $\Omega,$ then if we let $t\rightarrow\left\vert
\Omega\right\vert $ in
\[
f^{\ast\ast}(t)-f^{\ast}(t)\leq c_{n}t^{1/n}(\nabla f)^{\ast\ast}(t)
\]
we find that%
\begin{align*}
\frac{1}{\left\vert \Omega\right\vert }\int_{\Omega}\left\vert f(x)\right\vert
dx  &  =f^{\ast\ast}(\left\vert \Omega\right\vert )\leq\left\vert
\Omega\right\vert ^{1/n-1}\left\Vert \nabla f\right\Vert _{L^{1}(\Omega)}\\
&  \leq\left\Vert \nabla f\right\Vert _{L(n,q)}.
\end{align*}

\end{remark}

Let us consider the case $q=n.$ It was shown in \cite[page 1227]{7} (the so
called Hansson-Brezis-Wainger-Maz'ya embedding) that
\begin{align*}
\left\{  \int_{0}^{\left\vert \Omega\right\vert }\left(  \frac{f^{\ast\ast
}(s)}{1+\log\frac{\left\vert \Omega\right\vert }{s}}\right)  ^{n}\frac{ds}%
{s}\right\}  ^{1/n}  &  \preceq\left\Vert f\right\Vert _{L(\infty,n)(\Omega
)}\\
&  \preceq\left\Vert \nabla f\right\Vert _{L(n,q)}+\left\Vert f\right\Vert
_{L^{1}(\Omega)}.
\end{align*}
Therefore, (\ref{hbr}) implies an improvement on the
Hansson-Brezis-Wainger-Maz'ya embedding. The connection with $BMO$ appears
when $q=\infty,$ for then we have%
\[
\left\Vert f\right\Vert _{L(\infty,\infty)}\preceq\left\Vert \nabla
f\right\Vert _{L(n,\infty)},f\in C_{0}^{\infty}(R^{n}).
\]

In the case $p=n,q=1.$ Then, (\ref{sobol1}) gives
\begin{equation}
\left\Vert f\right\Vert _{L(\infty,1)}\preceq\left\Vert \nabla f\right\Vert
_{L(n,1)},f\in C_{0}^{\infty}(R^{n}), \label{comparada}%
\end{equation}
which ought to be compared with the following (cf. \cite{St1})%
\begin{equation}
\left\Vert f\right\Vert _{L^{\infty}}\preceq\left\Vert \nabla f\right\Vert
_{L(n,1)},f\in C_{0}^{\infty}(R^{n}). \label{comparada1}%
\end{equation}
Indeed, let us show that (\ref{comparada}) gives (\ref{comparada1}). From
\[
\frac{d}{dt}(tf^{\ast\ast}(t))=\frac{d}{dt}(\int_{0}^{t}f^{\ast}%
(s)ds)=f^{\ast}(t),
\]
it follows (by the product rule of Calculus) that%
\[
\frac{d}{dt}(f^{\ast\ast}(t))=-\left(  \frac{f^{\ast\ast}(t)-f^{\ast}(t)}%
{t}\right)  .
\]
Therefore, if $f$ has compact support$,$%
\begin{align*}
\left\Vert f\right\Vert _{L(\infty,1)}  &  =\lim_{t\rightarrow\infty}\int%
_{0}^{t}\left(  f^{\ast\ast}(s)-f^{\ast}(s)\right)  \frac{ds}{s}%
=\lim_{t\rightarrow\infty}\left(  f^{\ast\ast}(0)-f^{\ast\ast}(t)\right) \\
&  =\left\Vert f\right\Vert _{L^{\infty}}-\lim_{t\rightarrow\infty}\frac{1}%
{t}\left\Vert f\right\Vert _{L^{1}}\\
&  =\left\Vert f\right\Vert _{L^{\infty}}.
\end{align*}

\subsection{The Gagliardo-Nirenberg Inequality and Weak type vs Strong Type}

It is well known that the Sobolev inequalities have remarkable self improving
properties. In this section we wish to discuss the connections of these self
improving effects with symmetrization. The study is important when trying to
extend Sobolev inequalities to more general contexts.

We consider three forms of \ the Gagliardo-Nirenberg inequality. The strong
form of the Gagliardo-Nirenberg inequality%
\begin{equation}
\left\Vert f\right\Vert _{L(n^{\prime},1)}\preceq\left\Vert \nabla
f\right\Vert _{L^{1}},f\in C_{0}^{\infty}(R^{n}), \label{via}%
\end{equation}
which implies the classical version of the Gagliardo-Nirenberg inequality%

\begin{equation}
\left\Vert f\right\Vert _{L^{n^{\prime}}}\preceq\left\Vert \nabla f\right\Vert
_{L^{1}},f\in C_{0}^{\infty}(R^{n}). \label{via2}%
\end{equation}
which in turn implies the weaker version of the Gagliardo-Nirenberg inequality%
\begin{equation}
\left\Vert f\right\Vert _{L(n^{\prime},\infty)}\preceq\left\Vert \nabla
f\right\Vert _{L^{1}},f\in C_{0}^{\infty}(R^{n}). \label{via1}%
\end{equation}

Let us now show that (\ref{via1}) implies (\ref{via}). In \cite[(55) page
261]{A} we showed that (\ref{via1}) implies the symmetrization inequality%
\begin{equation}
f^{\ast\ast}(t)-f^{\ast}(t)\preceq t^{1/n}(\nabla f)^{\ast\ast}(t). \label{v2}%
\end{equation}
Conversely, (\ref{v2}) can be rewritten as%
\begin{equation}
(f^{\ast\ast}(t)-f^{\ast}(t))t^{1-1/n}\preceq\int_{0}^{t}(\nabla f)^{\ast
}(s)ds. \label{v3}%
\end{equation}
Consequently, taking a sup over all $t>0$ we see that (\ref{v2}) in turn
implies (\ref{via1}). Moreover, let us show that (\ref{v2}) implies the
isoperimetric inequality (here we ignore the issue of constants to simplify
the considerations). To see this suppose that $E$ is a bounded set with smooth
border and let $f_{n}$ be \ a sequence of smooth functions with compact
support such that $f_{n}\rightarrow\chi_{E}$ in $L^{1},$ with%
\[
\left\Vert \nabla f_{n}\right\Vert _{L^{1}}\rightarrow Per(E)
\]
where $Per(E)$ is the perimeter of $E.$ Selecting $t>\left\vert E\right\vert
,$ we see that $(f_{n}^{\ast\ast}(t)-f_{n}^{\ast}(t))\rightarrow\frac{1}%
{t}\left\vert E\right\vert ,$ therefore from (\ref{v3}) we find%
\[
\frac{1}{t}\left\vert E\right\vert t^{1-1/n}\preceq Per(E)
\]
therefore letting $t\rightarrow\left\vert E\right\vert ,$ gives%
\[
\left\vert E\right\vert ^{1-1/n}\preceq Per(E).
\]

This concludes our proof that (\ref{via1}) is equivalent to (\ref{via}) since
it is a well known consequence of the co-area formula that the isoperimetric
inequality is equivalent to (\ref{via}) (cf. \cite{67}). At the level of
symmetrization inequalities we have shown in \cite[page 263]{A} that
(\ref{via}) implies the symmetrization inequality%
\begin{equation}
\int_{0}^{t}(f^{\ast\ast}(s)-f^{\ast}(s))s^{1-1/n}\frac{ds}{s}\preceq\int%
_{0}^{t}(\nabla f)^{\ast}(s)ds. \label{v4}%
\end{equation}
Moreover, conversely, taking a sup over all $t>0$ in (\ref{v4})$,$ shows that
(\ref{v4}) implies (\ref{via}).

A direct proof of the fact that (\ref{v4}) implies (\ref{v2}) is
straightforward. Indeed, starting with%
\[
\int_{t/2}^{t}\left(  f^{\ast\ast}(s)-f^{\ast}(s)\right)  s^{1-1/n}\frac
{ds}{s}\preceq\int_{0}^{t}(\nabla f)^{\ast}(s)ds,
\]
and using the fact that $\left(  f^{\ast\ast}(t)-f^{\ast}(t\right)
)t=\int_{f^{\ast}(t)}^{\infty}\lambda_{f}(s)ds$ increases, we see that%
\[
\left(  f^{\ast\ast}(t/2)-f^{\ast}(t/2)\right)  t^{1-1/n}\preceq\int_{0}%
^{t}(\nabla f)^{\ast}(s)ds,
\]
and (\ref{v2}) follows readily. The proof that we give now, showing that
(\ref{v2}) implies (\ref{v4}) is indirect. First, as we have seen (\ref{v2})
is equivalent to the validity of (\ref{via1}) which in turn implies the
following inequality\footnote{Note that by P\'{o}lya-Szeg\"{o}, $f^{\ast}$ is
absolutely continuous} due to Maz'ya-Talenti (cf. \cite{65}),
\begin{equation}
t^{1-1/n}[-f^{\ast}(t)]^{\prime}\preceq\frac{d}{dt}(\int_{\{\left\vert
f(x)\right\vert >f^{\ast}(t)\}}\left\vert \nabla f(x)\right\vert
dx).\label{v5}%
\end{equation}
To proceed further we need a new expression for $f^{\ast\ast}(t)-f^{\ast}(t),$
which we derive integrating by parts:%
\begin{align}
f^{\ast\ast}(t)-f^{\ast}(t) &  =\frac{1}{t}\int_{0}^{t}[f^{\ast}(s)-f^{\ast
}(t)]ds\nonumber\\
&  =\frac{1}{t}\left.  (s[f^{\ast}(s)-f^{\ast}(t)])\right\vert _{s=0}%
^{s=t}+\frac{1}{t}\int_{0}^{t}s[-f^{\ast}(s)]^{\prime}ds\nonumber\\
&  =\frac{1}{t}\int_{0}^{t}s[-f^{\ast}(s)]^{\prime}ds.\label{numer}%
\end{align}
Therefore,%
\begin{align*}
\int_{0}^{t}\left(  f^{\ast\ast}(s)-f^{\ast}(s)\right)  s^{-1/n}ds &
=\int_{0}^{t}\frac{1}{s}\int_{0}^{s}u[-f^{\ast}(u)]^{\prime}dus^{-1/n}ds\\
&  =-n\int_{0}^{t}\left(  \int_{0}^{s}u[-f^{\ast}(u)]^{\prime}du\right)
ds^{-1/n}\\
&  =\left.  -n\left(  \int_{0}^{s}u[-f^{\ast}(u)]^{\prime}du\right)
s^{-1/n}\right\vert _{s=0}^{s=t}+n\int_{0}^{t}s[-f^{\ast}(s)]^{\prime}%
s^{-1/n}ds.
\end{align*}
We claim that we can discard the integrated term since its contribution makes
the right hand side smaller. To see this note that, since (\ref{v2}) holds,
(\ref{numer}) implies%
\[
\left(  \int_{0}^{s}u[-f^{\ast}(u)]^{\prime}du\right)  s^{-1/n}=\left(
f^{\ast\ast}(s)-f^{\ast}(s)\right)  s^{1-1/n}\preceq\int_{0}^{s}(\nabla
f)^{\ast}(u)du,
\]
which in turn implies that $\left(  \int_{0}^{s}u[-f^{\ast}(u)]^{\prime
}du\right)  s^{-1/n}\rightarrow0$ when $s\rightarrow0.$ Consequently, we can
continue our estimates to obtain,%
\begin{align*}
\int_{0}^{t}\left(  f^{\ast\ast}(s)-f^{\ast}(s)\right)  s^{-1/n}ds &  \preceq
n\int_{0}^{t}s[-f^{\ast}(s)]^{\prime}s^{-1/n}ds\\
&  \preceq\int_{0}^{t}[-f^{\ast}(s)]^{\prime}s^{1-1/n}ds\\
&  \preceq\int_{0}^{t}\frac{d}{dt}(\int_{\{\left\vert f(x)\right\vert
>f^{\ast}(s)\}}\left\vert \nabla f(x)\right\vert dx)ds\text{ \ \ (by(\ref{v5}%
))}\\
&  \leq\int_{\{\left\vert f(x)\right\vert >f^{\ast}(t)\}}\left\vert \nabla
f(x)\right\vert dx\\
&  \leq\int_{0}^{t}\left(  \nabla f\right)  ^{\ast}(s)ds.
\end{align*}

Underlying these equivalences between weak and strong inequalities is the
Maz'ya truncation principle (cf. \cite{34}) which, informally, shows that,
contrary to what happens for most other inequalities in analysis, in the case
of Sobolev inequalities: weak implies strong!

In \cite{A} we showed the connection of the truncation method to a certain
form of extrapolation of inequalities initiated by Burkholder and Gundy. The
import of these considerations is that the symmetrization inequalities hold in
a very general context and allow for some unification of Sobolev inequalities.
For example, the preceding analysis and the corresponding symmetrization
inequalities can be extended for gradients defined in metric measure spaces
using a variety of methods. One method, often favored by probabilists, goes
via defining the gradient by suitable limits, in this case, under suitable
assumptions, we can use isoperimetry to reformulate the symmetrization
inequalities and embeddings (cf. \cite{63}, \cite{64}, and the references
therein). In the context of metric probability spaces with concave
isoperimetric profile $I,$ the basic inequality takes the form%
\begin{equation}
f^{\ast\ast}(t)-f^{\ast}(t)\leq\frac{t}{I(t)}\left\vert \nabla f\right\vert
^{\ast\ast}(t).\label{v6}%
\end{equation}
For example, if we consider $R^{n}$ with Gaussian measure, the isoperimetric
profile satisfies%
\[
I(t)\sim t(\log\frac{1}{t})^{1/2},\text{ }t\text{ near zero.}%
\]
Thus in the Gaussian case (\ref{v6}) yields logarithmic Sobolev inequalities
(cf. \cite{Mm}, \cite{63}, \cite{64}, for more on this story). A somewhat
different approach, which yields however similar symmetrization inequalities,
obtains if we define the gradient indirectly via Poincar\'{e} inequalities and
then derive the symmetrization inequalities using maximal inequalities. The
analysis here depends an a large body of classical research on maximal
functions and Poincar\'{e} inequalities (for the symmetrization inequalities
that result we refer to \cite{47}, and Kalis' 2007 PhD thesis at FAU).

\end{document}